\documentclass[11pt]{article}

\usepackage[a4paper,total={7in,10.5in}]{geometry}
\usepackage[indent=16pt]{parskip}

\usepackage[T1]{fontenc}
\usepackage[utf8]{inputenc}
\usepackage[english]{babel}

\usepackage{epsfig}
\usepackage{epstopdf}

\usepackage{amsmath}
\usepackage{amsthm}
\usepackage{amsfonts}
\usepackage{amssymb}

\usepackage{tikz-cd}

\usepackage{xfrac}
\usepackage{hyperref}
\usepackage{comment}
\usepackage{ulem}
\usepackage{caption}
\usepackage{subcaption}
\usepackage{array}
\usepackage{ragged2e}
\usepackage{graphbox}

\renewcommand{\P}{{\mathbb P}}
\newcommand{\C}{{\mathbb C}}

\newcommand{\N}{{\mathbb N}}
\newcommand{\Q}{{\mathbb Q}}

\newcommand{\QQ}{{\overline {\mathbb Q}}}

\newcommand{\Gal}{\mathrm{Gal}}

\newcommand{\Sym}{\mathrm{\bf S}}

\newcommand{\sep}{\mathop{|}}
\renewcommand{\pmod}[1]{\ (\mathrm{mod}\ #1)}
\renewcommand{\le}{\leqslant}
\renewcommand{\ge}{\geqslant}

\newtheorem{theorem}{Theorem}[section]

\newtheorem{proposition}[theorem]{Proposition}
\newtheorem{conjecture}[theorem]{Conjecture}
\theoremstyle{definition}

\newtheorem{example}[theorem]{Example}
\newtheorem{remark}[theorem]{Remark}
\newtheorem{algorithm}[theorem]{Algorithm}

\usepackage{tcolorbox}
\usepackage{xcolor}

\usepackage{tcolorbox}
\usepackage{xcolor}

\newtcolorbox{webquote}{
  colback=black!3,       
  colframe=gray!60,      
  leftrule=4pt,          
  rightrule=0pt,         
  toprule=0pt,           
  bottomrule=0pt,        
  sharp corners,         
  left=12pt,             
  right=12pt,            
  top=8pt,               
  bottom=8pt,            
  before=\vspace{10pt},  
  after=\vspace{10pt}    
}

\begin{document}

\title{Almost Regular Coverings of the Sphere:
\\
Realizability. I. Tetrahedral Case}
\author{
Nikolai M. Adrianov\footnote{
Lomonosov Moscow State University, 119991 Moscow, Russia,
e-mail: {\tt nadrianov@gmail.com}},    
Elena M. Kreines\footnote{Department of Mathematics, Faculty of Natural Sciences, Ben-Gurion University of the Negev, P.O.B. 653, Beer Sheva, 8410501, Israel,
e-mail: {\tt kreines@tauex.tau.ac.il}}
}

\maketitle

\begin{abstract}
We prove the realizability of genus-$0$ branch data of the form $(2^r,a \sep 3^s,b \sep 3^t,c)$ and $(2^r,a \sep 3^s \sep 3^t,b,c)$, where $a$, $b$, $c$ are not divisible by $2$, $3$, $3$ respectively. The proof uses an explicit combinatorial description of coverings of the sphere branched over $3$ points via dessins d'enfants. As a corollary, we establish realizability for a broader class of branch data with more critical values.
\end{abstract}

\section{Introduction}

\subsection{Partitions}

For a positive integer $n\in\N$, a tuple
$$
\lambda=(a_1,a_2,\ldots, a_m), \qquad \sum_{i=1}^m a_i = n,\quad a_i\in\N
$$
is called a {\it partition} of $n$, written $\lambda\vdash n$. The number of parts $m$ is the {\it length} of the partition, denoted $\ell(\lambda)$.

Parts are often assumed to be arranged in non-increasing order $a_i\ge a_{i+1}$, but we do not impose this convention here. When many parts are equal, we use exponential notation; for example,
$$
(\underbrace{3,\ldots, 3}_s,5,2,2) = (3^s, 5, 2^2).
$$

\subsection{The Hurwitz Realizability Problem}

Let $f:\,X\to\P^1(\C)$ be a branched covering of the Riemann sphere of degree $n$, branched over $q$ points $x_1, \ldots, x_q\in\P^1(\C)$. The ramification over each point $x_i$ defines a partition $\lambda_i\vdash n$, so we associate to $f$ its branch datum
\begin{equation}
\pi(f)=(\lambda_1 \sep \lambda_2 \sep \ldots \sep \lambda_q),
\qquad
\lambda_i\vdash n,
\quad
\lambda_i\ne (1^n).
\label{eq:passport}
\end{equation}

The genus $g=g(X)$ of the covering is determined by the Riemann--Hurwitz formula:
\begin{equation}
\sum_{i=1}^q \ell(\lambda_i) - (q-2)n = 2 - 2g.
\label{eq:RH}
\end{equation}

A collection of partitions $\pi$ satisfying~(\ref{eq:RH}) for some integer $g\ge 0$ is called a {\it branch datum of genus $g$}.
A covering $f:\,X\to\P^1(\C)$ of genus $g$ whose branch datum equals $\pi$ is called a {\it realization} of $\pi$.
The Riemann--Hurwitz formula is a necessary but not sufficient condition for realizability. The realizability problem has a rich literature; we highlight the key paper of Edmonds, Kulkarni, and Stong~\cite{EKS-1984} and the computational results of Zheng~\cite{Zheng-2006}. For a detailed survey of the Hurwitz problem, see~\cite{Petr-2020} and the introduction to~\cite{Pak-2024}. Necessary conditions for realizability of various special families of branch data have been found, and in some cases their sufficiency has been proved; see~\cite{
BarPetr-2024,
KLN-2018,
MengWeiZhou-2026,
Pak-2009,
PakZvo-2014,
PascPetr-2012,
PervPetr-2006,
PervPetr-2008,
Thom-1965,
WeiWuXu-2024}.

In the present paper we prove the realizability of an extremal family of genus-$0$ branch data, providing a general scheme for a combinatorial construction via dessins d'enfants that can be applied to prove realizability of other related families as well.

The main result of this paper is the following theorem. Since almost all ramification indices in the first three components of the branch data equal $2$, $3$, and $3$, we call such branch data {\it almost regular of $(2,3,3)$ type} (the tetrahedral case). In subsequent papers we will prove analogous results for the octahedral $(2,3,4)$ and icosahedral $(2,3,5)$ types, and describe some new non-realizable families of genus-$0$ branch data.

\begin{theorem}
\label{thm:main}
Let a genus-$0$ branch datum have one of the following forms:
\begin{subequations}
\begin{equation}
\pi = (2^r,a \sep 3^s, b \sep 3^t,c \sep \lambda_4 \sep \ldots \sep \lambda_q),
\label{eq:passport-wide-233-a}
\end{equation}
or
\begin{equation}
\pi = (2^r,a \sep 3^s \sep 3^t,b,c \sep \lambda_4 \sep \ldots \sep \lambda_q),
\label{eq:passport-wide-233-b}
\end{equation}
\end{subequations}
where $2\nmid a$, $3\nmid b$, $3\nmid c$, and $\lambda_i$ are arbitrary partitions.
Then the branch datum $\pi$ is realizable.
\end{theorem}

\subsection{Dessins d'Enfants}

A covering $\beta:\,X\to\P^1(\C)$ branched over exactly $q=3$ points is called a {\it Belyi cover} (or {\it Belyi function}).

Choose any two of the three critical values, say $c_0,c_1\in \P^1(\C)$. The preimage of the arc connecting them yields a connected bicolored graph $\Gamma=\beta^{-1}[c_0,c_1]$ on the surface $X$ such that the complement $X\setminus \Gamma$ is a disjoint union of discs (faces). In other words, the pair $(X,\Gamma)$ is a bicolored map on a compact oriented surface. The preimages of $c_0$ are colored white, and those of $c_1$ are colored black.

We call the bicolored map $(X,\Gamma)$ a {\it dessin d'enfant}, a term coined by Grothendieck~\cite{Gro-1984}.
The use of this term, rather than the more prosaic ``map'', is motivated by a theorem of Belyi~\cite{Bel-1979}: the surface $X$ can be interpreted as a nonsingular complex algebraic curve defined over a number field (a finite extension of $\Q$). Conversely, every such curve admits a Belyi function, so dessins d'enfants provide a complete combinatorial encoding of curves over number fields. Furthermore, there is a faithful action of the absolute Galois group $\Gal(\QQ/\Q)$ on dessins d'enfants. In the present paper, however, we restrict ourselves to the combinatorial aspects and do not pursue arithmetic questions.
The theory of dessins d'enfants is a well-developed subject; for a detailed and accessible introduction we recommend the monograph~\cite{LanZvo-2004}.

The degree of a vertex or face equals the ramification index of the covering at that point. The {\it degree of a vertex} is the number of edges incident to it. The definition of face degree requires some care. Each face of a dessin can be represented as a polygon with an even number of sides whose vertices alternate in color, so the {\it degree of a face} is defined as {\it half} the number of edges incident to it.

In the theory of dessins d'enfants, the collection of vertex and face degrees is traditionally called the {\it passport}. In this paper we use the term ``passport'' as a synonym for ``branch datum of width $q=3$'' when speaking of a dessin.

The ordered pair of critical values can be chosen in $6$ ways (reflecting the action of $\Sym_3$ on the three critical points), so each Belyi cover can be depicted by $1$, $2$, $3$, or $6$ different dessins. With a slight abuse of terminology, we say that these dessins are {\it dual} to one another.
The passports of dual dessins are obtained as all possible permutations of the passport components:
\begin{center}
\begin{tabular}{lclcl}
$(\lambda_1 \sep \lambda_2 \sep \lambda_3)$, &&
$(\lambda_2 \sep \lambda_3 \sep \lambda_1)$, &&
$(\lambda_3 \sep \lambda_1 \sep \lambda_2)$, \\
$(\lambda_1 \sep \lambda_3 \sep \lambda_2)$, &&
$(\lambda_3 \sep \lambda_2 \sep \lambda_1)$, &&
$(\lambda_2 \sep \lambda_1 \sep \lambda_3)$. \\
\end{tabular}
\end{center}

In this paper we restrict to the case $g(X)=0$, so we work with connected graphs embedded in the sphere. Via stereographic projection we may view them as planar graphs, but a further choice remains: which face (or vertex) to place at $\infty\in\P^1(\C)$. We usually take the face of largest degree as the ``outer'' face, though this is not uniquely determined when several faces share the maximum degree, and we may depart from this convention for other reasons.

\begin{example}
Figure~\ref{fig:dessin-ex}$(a)$ shows a genus-$0$ dessin, one of two realizations of the passport $(2^3,1 \sep 3^2,1 \sep 6,1)$; the second realization is its mirror image. When almost all white vertices have degree $2$, we suppress them in the picture, see Fig.~\ref{fig:dessin-ex}$(b)$. Nevertheless, our dessins are always bicolored: if an edge appears to connect two black vertices, it is understood that this edge carries a white vertex of degree $2$.
\end{example}

\begin{figure}[h]
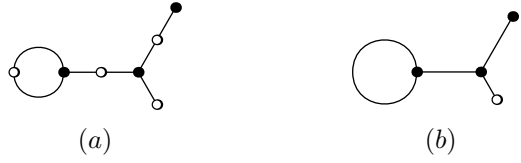

\centering
\begin{subfigure}[b]{0.25\textwidth}
\centering
\includegraphics{images/wt-behold-10.mps}
\caption*{$(a)$}
\end{subfigure}
\begin{subfigure}[b]{0.25\textwidth}
\centering
\includegraphics{images/wt-behold-11.mps}
\vspace{0.065cm}
\caption*{$(b)$}
\end{subfigure}
\caption{A 7-edge dessin with passport $(2^3,1 \sep 3^2,1 \sep 6,1)$.}
\label{fig:dessin-ex}
\end{figure}

Plane trees, i.e.\ dessins with a single face, form a special subclass of genus-$0$ dessins. The corresponding Belyi functions are polynomials (so-called Shabat polynomials). For plane trees, the Hurwitz realizability problem has a trivial answer: every plane tree passport is realizable. This follows from Tutte's explicit formula~\cite{Tutte-1964} for the number of plane trees with a given passport, and is also a special case of Thom's result~\cite{Thom-1965} on the realizability of branch data of polynomials.

The following theorem is the restriction of Theorem~\ref{thm:main} to the case of three critical values.

\begin{theorem}
\label{thm:main-dessin}
\begin{subequations}
For every genus-$0$ passport of the form
\begin{equation}
\pi = (2^r,a \sep 3^s, b \sep 3^t,c),
\label{eq:passport-233-a}
\end{equation}
or
\begin{equation}
\pi = (2^r,a \sep 3^s \sep 3^t,b,c),
\label{eq:passport-233-b}
\end{equation}
\end{subequations}
where $2\nmid a$, $3\nmid b$, $3\nmid c$, there exists a dessin d'enfant realizing it.
\end{theorem}

In Section~\ref{section:reduction} we reduce Theorem~\ref{thm:main} to the case $q=3$, and then, in Sections~\ref{section:realize-a} and \ref{section:realize-b}, describe a combinatorial construction that, given a passport of the form~(\ref{eq:passport-233-a}) or~(\ref{eq:passport-233-b}), produces a dessin realizing it.

\subsection{Uniqueness of the Dessin}

For $q=3$ it is natural to ask how many realizations a given passport admits.
The following conjecture is supported by our computer calculations.

\begin{conjecture}
\label{conj:main}
For each passport in the statement of Theorem~$\ref{thm:main-dessin}$, the realization as a dessin d'enfant is unique.
\end{conjecture}

Thus, the families of dessins constructed in this paper extend results previously obtained for plane trees~\cite{Adr-2007} and weighted trees~\cite{PakZvo-2014}. Note the structural similarity between our dessins and the weighted trees of series $C$ and $E$ in the classification of weighted unitrees due to Pakovich and Zvonkin; see~\cite{PakZvo-2014} or \cite[Ch.~5]{AdrPakZvo-2020}.

Although we do not compute the Belyi functions explicitly and do not address the action of the Galois group on dessins d'enfants, we remark that Conjecture~\ref{conj:main} would imply that the corresponding Belyi covers can be defined over $\Q$. Furthermore, one can show that these functions cannot be expressed as nontrivial compositions. Compositionally irreducible Belyi functions over $\Q$ are interesting and rather rare objects.

\section{Proofs without Words: Historical Remarks}

\subsection{Bh\=askara: Behold!}

A popular story, found in some histories of mathematics and widely repeated in mathematical folklore, has it that the Indian mathematician Bh\=askara~II (12th century AD) proved the Pythagorean theorem in his treatise Bijaganita using a single diagram accompanied by the single word ``Behold!''; see~\cite{Caj-1910}:

\begin{webquote}
Arranging this square and the four triangles in a different way, they are seen, together,
to make up the sum of the square of the two sides. ``Behold!'' says Bhaskara,
without adding another word of explanation.
\end{webquote}

In fact, this is
something of an exaggeration: Bh\=askara uses the word much as a modern author uses ``see'' to refer the reader to an illustrating figure; see~\cite{Plo-2007}, p.~477, or~\cite{Caj-1910}, p.~333, Appendix~11 of the Russian edition of 1910.

Nevertheless, visual proofs without words are immensely attractive, and the beautiful legend gives them a perfect setting.

\subsection{Visual Proofs via Dessins d'Enfants}

A striking example of a Behold!-style proof using dessins d'enfants was given by Pakovich and Zvonkin; see~\cite{PakZvo-2014} or~\cite[Ch.~2]{AdrPakZvo-2020}.

The conjecture on the existence of coprime polynomials $A,B\in\C[z]$ with $\deg A=2k$, $\deg B=3k$ such that $\deg (A^3-B^2) = k+1$, for all $k\in \N$, was stated by Birch, Chowla, Hall, and Schinzel~\cite{BCHS-1965} in 1965; its proof was given by Stothers~\cite{Sto-1981} in 1981 using group-theoretic arguments.

\begin{figure}[h]
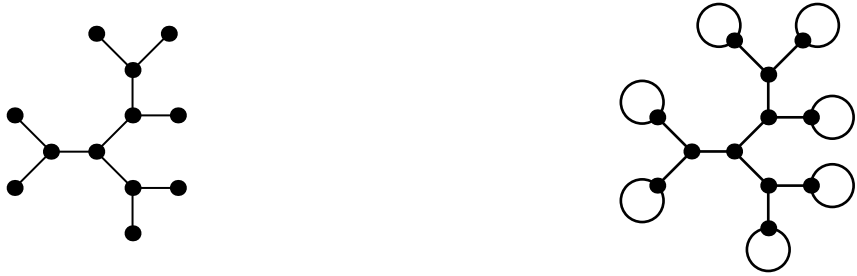

\centering
\begin{subfigure}[b]{0.4\textwidth}
\centering
\includegraphics{images/wt-behold-1.mps}
\vspace{0.41cm}
\caption*{$(a)$ Take a plane tree with $k+1$ leaves and $k-1$ internal vertices of degree $3$.}
\end{subfigure}
\hspace{0.6cm}
\begin{subfigure}[b]{0.45\textwidth}
\centering
\includegraphics{images/wt-behold-2.mps}
\caption*{$(b)$ Attach a loop to each leaf, obtaining a dessin with $2k$ black vertices of degree $3$.}
\end{subfigure}
\caption{The Pakovich--Zvonkin proof of the theorem on the minimum degree of $A^3-B^2$.}
\label{fig:DZ-dessin}
\end{figure}

Several pages of Stothers' proof can be replaced by a single picture, Fig.~\ref{fig:DZ-dessin}. The Belyi function of the dessin on the right has the form $\beta = A^3/R = B^2/R + 1$, and the polynomial $R=A^3-B^2$ has degree $k+1$, where $k+1$ is the number of loops in the dessin. For any given $k$, such a dessin exists and the construction is clear from the picture. Behold!

In fact, the solution of the $A^3-B^2$ conjecture over $\C$ is equivalent to the realizability of branch data of the form
$$
\left(2^{3k} \sep 3^{2k} \sep (5k-1),1^{k+1}\right).
$$

Our proof of Theorem~\ref{thm:main-dessin} is more verbose, but in essence it consists of Figures~\ref{fig:233-generic-T}--\ref{fig:233-generic-NT-2} and~\ref{fig:233-generic2-T}--\ref{fig:233-generic2-NT-2}, with all the necessary accompanying text contained in the descriptions of Algorithms~\ref{algo:main} and~\ref{algo:main2}. Behold!

\subsection{Bh\=askara: Non-Figures}

In another treatise, L\=il\=avat\=i, Bh\=askara introduces the term ``non-figure'' to describe a nonexistent polygon whose sides fail to satisfy the (generalized) triangle inequality~\cite{Plo-2009}, p.~188:

\begin{webquote}
A straight-sided figure described in [over]-confidence
where the sum of [all] the sides except [one] is smaller
than or equal to [that] one side is to be known as a
``non-figure.'' (verse 163)
\end{webquote}

This resonates remarkably with our interest in the Hurwitz realizability problem. Perhaps we should adopt the term ``non-cover'' for non-realizable branch data?

Paying homage to Bh\=askara, in this paper we use the term ``non-triangle'' for a triple $(u,v,w)$ that fails to satisfy the triangle inequality.

\section{Reduction to Three Critical Values}
\label{section:reduction}

\begin{proposition}
If the realizability of branch data of the form~$(\ref{eq:passport-wide-233-a})$ or $(\ref{eq:passport-wide-233-b})$ holds for $q=3$, then it holds for all $q$.
\end{proposition}

\proof We argue by induction, with the case $q=3$ as the base.
Let $q\ge 4$ and assume every such passport of width $q-1$ is realizable. Take an arbitrary genus-$0$ passport $\pi$ of the form~(\ref{eq:passport-wide-233-a}) or~(\ref{eq:passport-wide-233-b}).

Write $\lambda=\lambda_q = 1^{m_1}\, 2^{m_2}\, 3^{m_3} \ldots$; then
$$
n = m_1 + 2m_2 + 3m_3 + \ldots
\ge 2 (m_1 + m_2 + m_3 + \ldots) - m_1
= 2 \ell(\lambda) - m_1,
$$
$$
\ell(\lambda) \le \frac{n+m_1}{2}.
$$
Set $m := r + \ell(\lambda) - n$. From the Riemann--Hurwitz formula,
\begin{align*}
2 &= r + 1 + s + 1 + t + 1 - n + \sum_{i=4}^{q}(\ell(\lambda_i) - n) \le {} \\
&\le r +  s + t + 3 - n + \ell(\lambda) - n = {} \\
&= m + s + t + 3 - n \le m + \frac{2n-2}{3} + 3 - n,
\end{align*}
so $m \ge (n-1)/3 > 0$. Moreover, $r = (n-a)/2 < n/2$, and therefore
$$
m < \frac{n}{2} + \frac{n+m_1}{2} - n = \frac{m_1}{2},
$$
so the number of $1$'s in $\lambda$ exceeds $2m$. Consequently, we can split $\lambda$ into two parts:
$$
\lambda = \lambda' \cup (1^{2m}), \qquad \lambda' \vdash (n-2m), \qquad \ell(\lambda') = \ell(\lambda) - 2m.
$$

Consider the passport
$$
\pi' = (2^m, (n-2m) \sep 3^s, b \sep 3^t, c \sep \lambda_4 \sep \ldots \sep \lambda_{q-1} ).
$$
By the Riemann--Hurwitz formula,
\begin{align*}
2 - 2g(\pi') &= m + 1 + s + 1 + t + 1 + \sum_{i=4}^{q-1} \ell(\lambda_i) -(q-3) n = {}\\
&= r +\ell(\lambda) - n + 1  + s + 1 + t + 1 + \sum_{i=4}^{q-1} \ell(\lambda_i) - (q-3) n = {}\\
&= r + 1 + s + 1 + t + 1 + \sum_{i=4}^{q} \ell(\lambda_i) - (q-2) n = 2 - 2g(\pi).
\end{align*}
Hence $g(\pi') = g(\pi) = 0$ and, by the induction hypothesis, $\pi'$ is realizable. By the monodromy correspondence going back to Hurwitz~\cite{Hur-1891}, there exist permutations $g_i\in\Sym_n$, $i=1,\ldots,q-1$, such that
\begin{itemize}
\item[(1)] the cycle types match the passport $\pi'$:
$$
\bigl(c(g_1),\ldots,c(g_{q-1})\bigr) = \pi';
$$
\item[(2)] their product is the identity: $g_1\ldots g_{q-1}=1$;
\item[(3)] they generate a transitive subgroup of $\Sym_n$.
\end{itemize}
Conjugating all $g_i$ by a suitable element of $\Sym_n$, we may assume that
$$
g_1 = (1,2,\ldots, 2n-m) \cdot \omega_1 \ldots \omega_m,
$$
where $\omega_i$ are disjoint transpositions not involving $1,2,\ldots,2n-m$.

From the definition $m=r+\ell(\lambda)-n$ it is clear that $m\le r$. Consider the passport $\pi'' = (\lambda' \sep 2^{r-m}, a \sep (n-2m))$. By the Riemann--Hurwitz formula its genus is zero:
\begin{align*}
2 - 2g(\pi'') &= \ell(\lambda') + r - m + 1 + 1 - (n-2m) = {}\\
&= \ell(\lambda) - 2m + r - m + 2 - (n - 2m) = {}\\
&= \ell(\lambda) + r - m + 2 - n = 2.
\end{align*}

Thus $\pi''$ is a plane tree passport and hence realizable: there exist permutations $h_1, h_2\in \Sym_{n-2m}$ with
$$
c(h_1) = \lambda', \qquad
c(h_2) = (2^{r-m},a), \qquad
h_1\cdot h_2 = (1,2,\ldots,n-2m).
$$
Set $h=h_2\cdot\omega_1\ldots \omega_m$; then
$$
c(h)=(2^m) \cup (2^{r-m},a) = (2^r,a),
$$
$$
h_1 \cdot h = h_1 \cdot h_2\cdot\omega_1\ldots \omega_m = g_1.
$$
The tuple $(h,g_2,\ldots,g_{q-1},h_1)$ defines a covering with passport $\pi$:
\begin{itemize}
\item[(1)] the cycle types match $\pi$;
\item[(2)] the product is the identity:
$$
h\cdot g_2\ldots g_{q-1}\cdot h_1 \sim h_1\cdot h\cdot g_2\ldots g_{q-1} = g_1g_2\ldots g_{q-1} = 1;
$$
\item[(3)] the permutations generate a group containing $g_1,\ldots,g_{q-1}$, hence a transitive group.
\end{itemize}
This completes the proof of the proposition.
\hfill $\square$

\section{Realizability of $(2^r,a \sep 3^s,b \sep 3^t,c)$}
\label{section:realize-a}

\subsection{Dessins for $a=1$, $b=1,2$.}

Figure~\ref{fig:233-dessins} shows all genus-$0$ dessins with $a=1$, $b=1,2$ up to $n\le 25$.

\begin{figure}[htp]
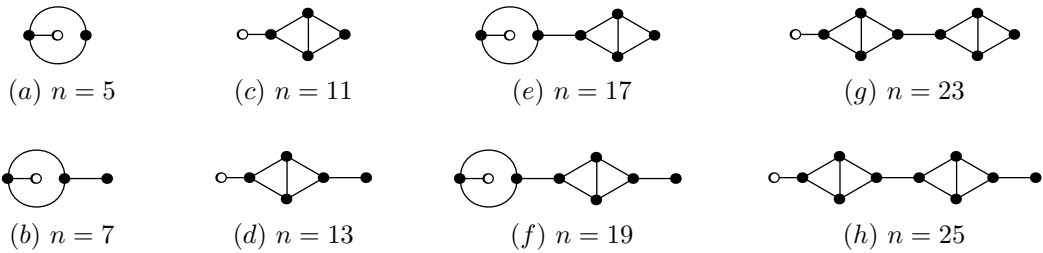

\centering
\begin{subfigure}[b]{0.15\textwidth}
\centering
\includegraphics[angle=180]{images/233_dessins-5.mps}
\vspace{-0.05cm}
\caption*{$(a)$ $n=5$}
\end{subfigure}
\begin{subfigure}[b]{0.18\textwidth}
\centering
\includegraphics[angle=180]{images/233_dessins-11.mps}
\caption*{$(c)$ $n=11$}
\end{subfigure}
\begin{subfigure}[b]{0.22\textwidth}
\centering
\includegraphics[angle=180]{images/233_dessins-17.mps}
\vspace{-0.05cm}
\caption*{$(e)$ $n=17$}
\end{subfigure}
\begin{subfigure}[b]{0.26\textwidth}
\centering
\includegraphics[angle=180]{images/233_dessins-23.mps}
\caption*{$(g)$ $n=23$}
\end{subfigure}
\vskip 0.5cm
\begin{subfigure}[b]{0.15\textwidth}
\centering
\includegraphics[angle=180]{images/233_dessins-7.mps}
\vspace{-0.05cm}
\caption*{$(b)$ $n=7$}
\end{subfigure}
\begin{subfigure}[b]{0.18\textwidth}
\centering
\includegraphics[angle=180]{images/233_dessins-13.mps}
\caption*{$(d)$ $n=13$}
\end{subfigure}
\begin{subfigure}[b]{0.22\textwidth}
\centering
\includegraphics[angle=180]{images/233_dessins-19.mps}
\vspace{-0.05cm}
\caption*{$(f)$ $n=19$}
\end{subfigure}
\begin{subfigure}[b]{0.26\textwidth}
\centering
\includegraphics[angle=180]{images/233_dessins-25.mps}
\caption*{$(h)$ $n=25$}
\end{subfigure}
\caption{Dessins with passport $(2^r,a \mid 3^s,b \mid 3^t,c)$ for $a=1$ and $b=1,2$.}
\label{fig:233-dessins}
\end{figure}

This picture suggests the construction for arbitrary $n$.


\subsection{Assembling a Dessin: Terminals and Connectors}

A fairly broad class of genus-$0$ dessins with passports of the form $(2^r,a \mid 3^s,b \mid 3^t,c)$ can be described as follows. Take two partial graphs, which we call {\it terminals}, and connect them by a chain of other partial graphs, which we call {\it connectors}.

The terminals $A^k_1$, $A^k_2$, $B^k_1$, $B^k_2$ ($k\ge 0$) and the connector $C$ are shown in Fig.~\ref{fig:233-terminals}. Terminals $A^k_1$, $A^k_2$ contain one white vertex whose degree differs from $2$: that degree equals $2k+1$. Terminals $B^k_1$, $B^k_2$ contain one black vertex whose degree differs from $3$: that degree equals $3k+1$ and $3k+2$ respectively.

\begin{figure}[!ht]
\centering
\begin{subfigure}[b]{0.4\textwidth}
\begin{center}
\begin{tabular}{ccc}
\vspace{-0.97cm}
\begin{subfigure}[b]{0.2\textwidth}
\hspace{0.5cm}
\includegraphics[angle=180]{images/233_terminals-11.mps}
\vspace{0.44cm}
\caption*{$A_1^0$}
\end{subfigure}
&
\hspace{0.6cm}
&
\begin{subfigure}[b]{0.2\textwidth}
\includegraphics[angle=180]{images/233_terminals-21.mps}
\caption*{$A_2^0$}
\end{subfigure}
\\[1cm]
\begin{subfigure}[b]{0.2\textwidth}
\includegraphics[angle=180]{images/233_terminals-13.mps}
\vspace{0.55cm}
\caption*{$A_1^1$}
\end{subfigure}
&
&
\begin{subfigure}[b]{0.2\textwidth}
\includegraphics[angle=180]{images/233_terminals-23.mps}
\vspace{0.09cm}
\caption*{$A_2^1$}
\end{subfigure}
\\[0.3cm]
\begin{subfigure}[b]{0.2\textwidth}
\includegraphics[angle=180]{images/233_terminals-15.mps}
\vspace{0.2cm}
\caption*{$A_2^2$}
\end{subfigure}
&
&
\begin{subfigure}[b]{0.2\textwidth}
\includegraphics[angle=180]{images/233_terminals-25.mps}
\vspace{0.3cm}
\caption*{$A_2^2$}
\end{subfigure}
\end{tabular}
\end{center}
\vspace{-0.2cm}
\caption*{$(a)$}
\end{subfigure}
\begin{subfigure}[b]{0.1\textwidth}
\centering
\begin{subfigure}[b]{\textwidth}
\centering
\includegraphics[angle=180]{images/233_terminals-99.mps}
\vspace{0.1cm}
\caption*{$C$}
\end{subfigure}
\vspace{6.63cm}
\end{subfigure}
\begin{subfigure}[b]{0.4\textwidth}
\begin{center}
\begin{tabular}{ccc}
\begin{subfigure}[b]{0.2\textwidth}
\includegraphics[angle=180]{images/233_terminals-31.mps}
\vspace{0.4cm}
\caption*{$B_1^0$}
\end{subfigure}
&
\hspace{0.4cm}
&
\begin{subfigure}[b]{0.2\textwidth}
\includegraphics[angle=180]{images/233_terminals-32.mps}
\vspace{0.1cm}
\caption*{$B_2^0$}
\end{subfigure}
\\[0.4cm]
\begin{subfigure}[b]{0.2\textwidth}
\includegraphics[angle=180]{images/233_terminals-34.mps}
\vspace{0.45cm}
\caption*{$B_1^1$}
\end{subfigure}
&
&
\begin{subfigure}[b]{0.2\textwidth}
\includegraphics[angle=180]{images/233_terminals-35.mps}
\vspace{0.45cm}
\caption*{$B_2^1$}
\end{subfigure}
\\[0.23cm]
\begin{subfigure}[b]{0.2\textwidth}
\includegraphics[angle=180]{images/233_terminals-37.mps}
\vspace{0.14cm}
\caption*{$B_1^2$}
\end{subfigure}
&
&
\begin{subfigure}[b]{0.2\textwidth}
\includegraphics[angle=180]{images/233_terminals-38.mps}
\vspace{0.14cm}
\caption*{$B_2^2$}
\end{subfigure}
\end{tabular}
\end{center}
\vspace{-0.2cm}
\caption*{$(b)$}
\end{subfigure}
%
\caption{%
        \begin{tabular}[t]{@{} l @{\ } >{\RaggedRight}p{0.6\textwidth} @{}}
            $(a)$ & $(2,3,3)$-terminals for $a\equiv 1\pmod{2}$ (white vertex); \\
            $(b)$ & $(2,3,3)$-terminals for $b\equiv 1,2\pmod{3}$ (black vertex).
        \end{tabular}%
    }
\label{fig:233-terminals}
\end{figure}

Note that both the connector $C$ and the repeating blocks in terminals $A_1^i$, $B_1^j$, $B_2^j$ are essentially tetrahedral graphs. For the terminal $A_2^i$ the addition of a tetrahedron is slightly less obvious: cut along a radial edge in $A_2^0$, cut along an edge of the tetrahedron, and glue.

Consequently, varying the number of connectors or the number of repeating blocks in the terminals changes the edge count of the graph by a multiple of $12$.

Using the connectors and terminals described above, we can construct a dessin with any prescribed $a$, $b$, and sufficiently large $c$; see the general picture in Fig.~\ref{fig:233-generic-NT-3}. It remains to handle the cases not covered by this construction.

\subsection{Constructing the Dual Dessin}

In this section we describe the realization of a dessin for sufficiently large values of $a$.

\begin{figure}[ht]
\centering
\centering
\includegraphics[align=c]{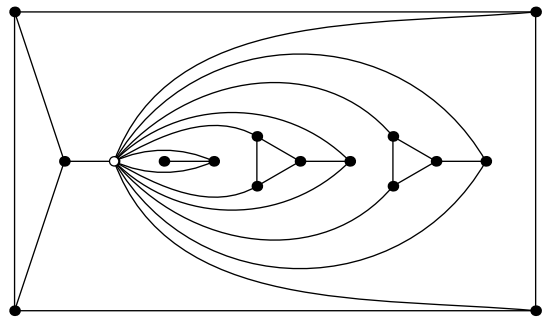}
\caption{A dessin with passport $(2^{15}, 13 \sep 3^{14}, 1 \sep 3^{13}, 4)$.}
\label{fig:332-dessin-ex0}
\end{figure}

Figure~\ref{fig:332-dessin-ex0} shows a $43$-edge dessin realizing the passport $(2^{15}, 13\sep 3^{14}, 1 \sep 3^{13}, 4)$. A certain structure is visible in this dessin from which further similar dessins can be derived, but describing it directly is awkward.

Instead, we describe the construction of a dessin with passport $(3^s,b\sep 3^t,c\sep 2^r,a)$; the dessin we need is then obtained by taking the dual.

As in the $(2,3,3)$ case, some dessins with passports $(3^s,b\sep 3^t,c\sep 2^r,a)$ can be realized by connecting two terminals (one with a white vertex of degree $b$, one with a black vertex of degree $c$) either by a single edge or by a chain of connectors. The only possible connector is shown in Fig.~\ref{fig:332-connector}, and the terminals form four series $D^k_q$, $q=0,3,8,11$, depicted in Figs.~\ref{fig:332-terminals-0-11} and~\ref{fig:332-terminals-3-8}. In these figures, every vertex of the terminal $D^k_q$ has degree $3$ except for one black vertex of degree $3k+1$ or $3k+2$. Swapping the colors of all vertices of $D^k_q$ gives the terminal $\bar D^k_q$, which has one white vertex of degree $3k+1$ or $3k+2$.

\begin{figure}[p]
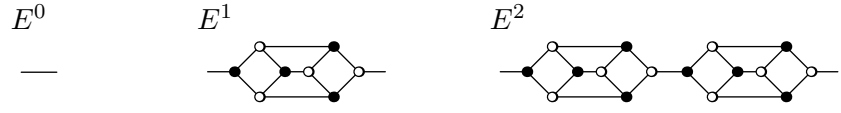

\centering
\begin{tabular}{ccccc}
\rlap{\raisebox{2\totalheight}{$E^0$}}
\includegraphics[align=c]{images/332_terminals-0.mps}
&
\hspace{1cm}
&
\rlap{\raisebox{2\totalheight}{$E^1$}}
\includegraphics[align=c]{images/332_terminals-1.mps}
&
\hspace{0.5cm}
&
\rlap{\raisebox{2\totalheight}{$E^2$}}
\includegraphics[align=c]{images/332_terminals-2.mps}
\end{tabular}
\caption{The $(3,3,2)$-connector and its repetitions.}
\label{fig:332-connector}
\end{figure}

\begin{figure}[p]
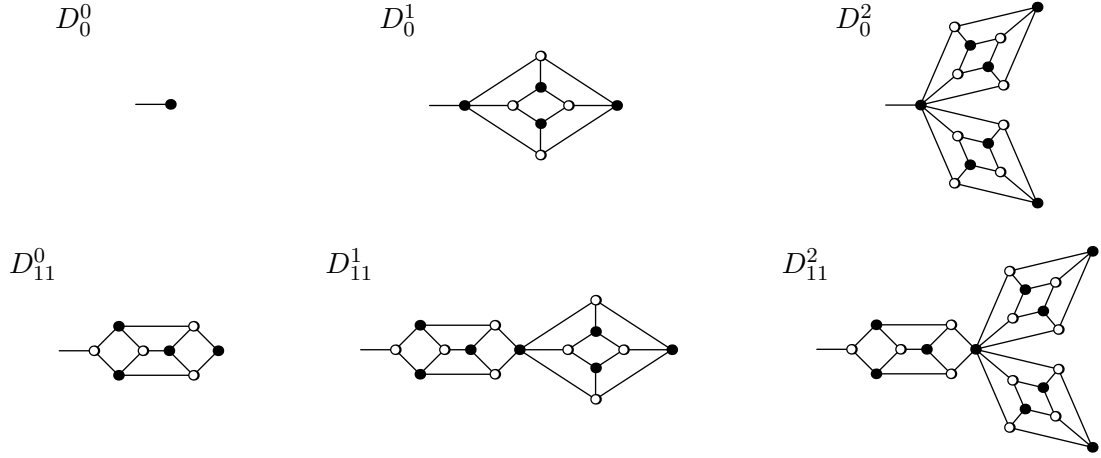

\centering
\begin{tabular}{ccccc}
\rlap{\raisebox{2.6\totalheight}{$D^0_0$}}
\hspace{0.8cm}
\includegraphics[align=c]{images/332_terminals-10.mps}
&
\hspace{0.5cm}
&
\rlap{\raisebox{2.6\totalheight}{$D^1_0$}}
\hspace{0.4cm}
\includegraphics[align=c]{images/332_terminals-11.mps}
&
\hspace{0.5cm}
&
\rlap{\raisebox{2.6\totalheight}{$D^2_0$}}
\hspace{0.4cm}
\includegraphics[align=c]{images/332_terminals-12.mps}
\\
\\
\rlap{\raisebox{2.6\totalheight}{$D^0_{11}$}}
\hspace{0.4cm}
\includegraphics[align=c]{images/332_terminals-110.mps}
&
&
\rlap{\raisebox{2.6\totalheight}{$D^1_{11}$}}
\hspace{0.2cm}
\includegraphics[align=c]{images/332_terminals-111.mps}
&
&
\rlap{\raisebox{2.6\totalheight}{$D^2_{11}$}}
\hspace{0.2cm}
\includegraphics[align=c]{images/332_terminals-112.mps}
\\
\end{tabular}
\caption{$(3,3,2)$-terminals $D^k_0$, $D^k_{11}$.}
\label{fig:332-terminals-0-11}
\end{figure}

\begin{figure}[p]
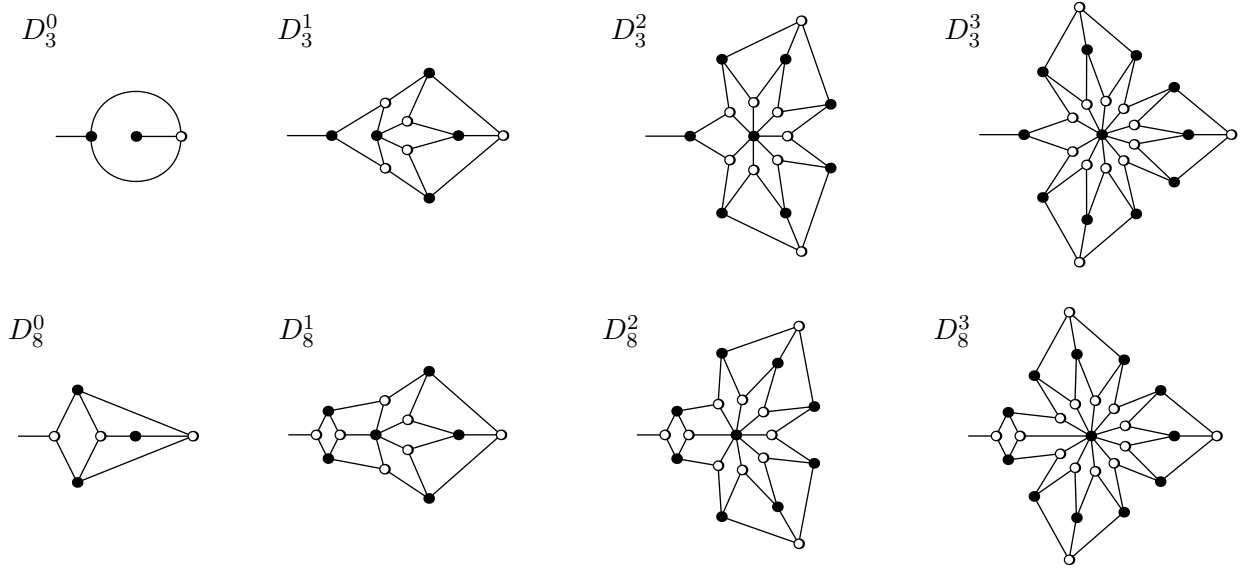

\centering
\begin{tabular}{ccccccc}
\rlap{\raisebox{3.2\totalheight}{$D^0_3$}}
\hspace{0.2cm}
\includegraphics[align=c]{images/332_terminals-30.mps}
&
\hspace{0.15cm}
&
\rlap{\raisebox{3.2\totalheight}{$D^1_3$}}
\includegraphics[align=c]{images/332_terminals-31.mps}
&
\hspace{0.15cm}
&
\hspace{0.2cm}
\rlap{\raisebox{3.2\totalheight}{$D^2_3$}}
\hspace{0.2cm}
\includegraphics[align=c]{images/332_terminals-32.mps}
&
\hspace{0.25cm}
&
\hspace{0.2cm}
\rlap{\raisebox{3.2\totalheight}{$D^3_3$}}
\hspace{0.2cm}
\includegraphics[align=c]{images/332_terminals-33.mps}
\\
\\
\rlap{\raisebox{3.2\totalheight}{$D^0_8$}}
\includegraphics[align=c]{images/332_terminals-80.mps}
&
&
\rlap{\raisebox{3.2\totalheight}{$D^1_8$}}
\includegraphics[align=c]{images/332_terminals-81.mps}
&
&
\rlap{\raisebox{3.2\totalheight}{$D^2_8$}}
\hspace{0.2cm}
\includegraphics[align=c]{images/332_terminals-82.mps}
&
&
\rlap{\raisebox{3.2\totalheight}{$D^3_8$}}
\hspace{0.2cm}
\includegraphics[align=c]{images/332_terminals-83.mps}
\\
\end{tabular}
\caption{$(3,3,2)$-terminals $D^k_3$, $D^k_8$.}
\label{fig:332-terminals-3-8}
\end{figure}

\begin{figure}[p]
\centering
\includegraphics[align=c]{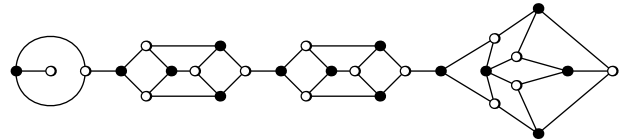}
\caption{A dessin with passport $(3^{14}, 1 \sep 3^{13}, 4 \sep 2^{15}, 13)$: $\bar D^0_3 + E^2 + D^1_3$.}
\label{fig:332-dessin-ex1}
\end{figure}

\begin{figure}[p]
\centering
\includegraphics[align=c]{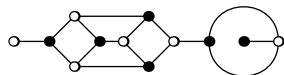}
\caption{A dessin with passport $(3^{5}, 1 \sep 3^{5}, 1 \sep 2^{5}, 6)$: $\bar D^0_0 + E^1 + D^0_3$.}
\label{fig:332-dessin-ex2}
\end{figure}

The number of edges in the terminal $D^k_q$ equals $12k+q$. Connecting terminals $D^{\lfloor b/3\rfloor}_q$ and $D^{\lfloor c/3\rfloor}_q$, we obtain a dessin with a white vertex of degree $b$, a black vertex of degree $c$, and total edge count $n\equiv 2q+1\pmod{12}$. Therefore, for each possible residue $n\equiv 1,5,7,11\pmod{12}$ we may take $q=0,8,3,11$ respectively.

An example of a dessin obtained by this construction is shown in Fig.~\ref{fig:332-dessin-ex1}; compare it with Fig.~\ref{fig:332-dessin-ex0}. All four series are illustrated in more detail in Fig.~\ref{fig:233-generic-NT-2} below (note, however, that $D^0_q$ looks quite different from $D^k_q$ for $k>0$).

\begin{remark}
\label{rem:negative-connector}
In the case $n\equiv 11 \pmod{12}$, the two terminals can not only be joined directly without any connector, but can also ``overlap'' each other; see the example in Fig.~\ref{fig:332-dessin-ex3}.
\end{remark}

\begin{figure}
\centering
\includegraphics[align=c]{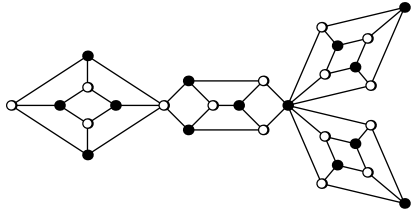}
\caption{Edge case: overlapping terminals $D^k_{11}$ (connector $E^{-1}$) for $n\equiv 11\pmod{12}$.}
\label{fig:332-dessin-ex3}
\end{figure}

\begin{remark}
In the construction above we connect two terminals with the same value of $q$. Can one use terminals with different values of $q$? For example, Fig.~\ref{fig:332-dessin-ex2} shows a dessin assembled from $\bar D^0_0$, one connector, and $D^0_3$. However, this dessin has $16$ edges and its outer face has degree $a=6$, so the condition $2\nmid a$ fails. The following table of values of $q_1+q_2+1\pmod{12}$ shows that mixing terminals with different parameters $q$ produces a dessin whose edge count is divisible by $2$ or $3$:
\begin{center}
\begin{tabular}{c|cccc}
& 0 & 3 & 8 & 11 \\
\hline
0 & 1 & 4 & 9 & 0 \\
3 & 4 & 7 & 0 & 3 \\
8 & 9 & 0 & 5 & 8 \\
11 & 0 & 3 & 8 & 11 \\
\end{tabular}
\end{center}
Nevertheless, we will use the combinations $D^i_0 + E^k + D^j_8$ and $D^i_3 + E^k + D^j_{11}$ to construct the dual dessin in the proof of part~$(b)$ of Theorem~\ref{thm:main-dessin}.
\end{remark}


\subsection{The Construction Algorithm}

We are now ready to describe the structure of the dessin in the general case. It turns out that the structure of a dessin with a passport of the form~(\ref{eq:passport-233-a}) depends on the residue $n \pmod{12}$ and on certain relations among $a$, $b$, and $c$.

Write $a = 2u+1$, $b = 3v + 1,2$, $c = 3w + 1,2$ with $u,v,w\ge 0$.
By the symmetry of the tetrahedral case, we may assume $v\le w$.

Depending on the residue of $n$ modulo $12$, the Riemann--Hurwitz formula gives:
\begin{itemize}
\item $n\equiv 1,7\pmod{12}$ $\Rightarrow$
$$
2 = (r+1) + (s+1) + (t+1) - n = \frac{n-a}{2} + \frac{n-b}{3} + \frac{n-c}{3} - n + 3 =
\frac{n-1}{6} - u - v - w + 2,
$$
$$
u + v + w = \frac{n-1}{6}.
$$
\item $n\equiv 5,11\pmod{12}$ $\Rightarrow$
$$
2 = (r+1) + (s+1) + (t+1) - n = \frac{n-a}{2} + \frac{n-b}{3} + \frac{n-c}{3} - n + 2 =
\frac{n-5}{6} - u - v - w + 2,
$$
$$
u + v + w = \frac{n-5}{6}.
$$
\end{itemize}
Consequently, $u+v+w$ is even when $n\equiv 1,5\pmod{12}$ and odd when $n\equiv 7,11\pmod{12}$. Define the ``integer semiperimeter''
$$
p =
\begin{cases}
\displaystyle
\frac{u+v+w}{2}, & \text{for}\ n\equiv 1,5\pmod{12};\\[4pt]
\displaystyle
\frac{u+v+w-1}{2}, & \text{for}\ n\equiv 7,11\pmod{12}.
\end{cases}
$$

The construction splits into three cases according to whether the triangle inequalities hold:
{\setlength{\leftmargini}{2.5cm}
\begin{itemize}
\item[(T)] Triangle case (triangle inequalities hold): $u\le p$, $v\le p$, $w\le p$.
\item[(NT-3)] Non-Triangle-3 case (third inequality fails): $w > p$.
\item[(NT-2)] Non-Triangle-2 case (first inequality fails): $u > p$.
\end{itemize}
}

\begin{remark}
When $n\equiv 1,5\pmod{12}$, the inequality $u\le p$ is trivially equivalent to $u\le v+w$. When $n\equiv 7,11\pmod{12}$, the inequality $u\le p$ is equivalent to $u+1\le v+w$, but since $u+v+w$ is odd, this is in turn equivalent to $u\le v+w$. Thus, despite the non-standard definition of the ``semiperimeter'' $p$, the stated inequalities are equivalent to the usual triangle inequalities.
\end{remark}

\begin{remark}
Our interpretation of $u$, $v$, $w$ as side lengths of a triangle (or of a non-triangle) is formal. We use it to unify the construction of the dessin realizing a given passport. It would be interesting to find a deeper connection between the geometry of this triangle and the analytic properties of the corresponding Belyi functions.
\end{remark}


\begin{figure}[ht]
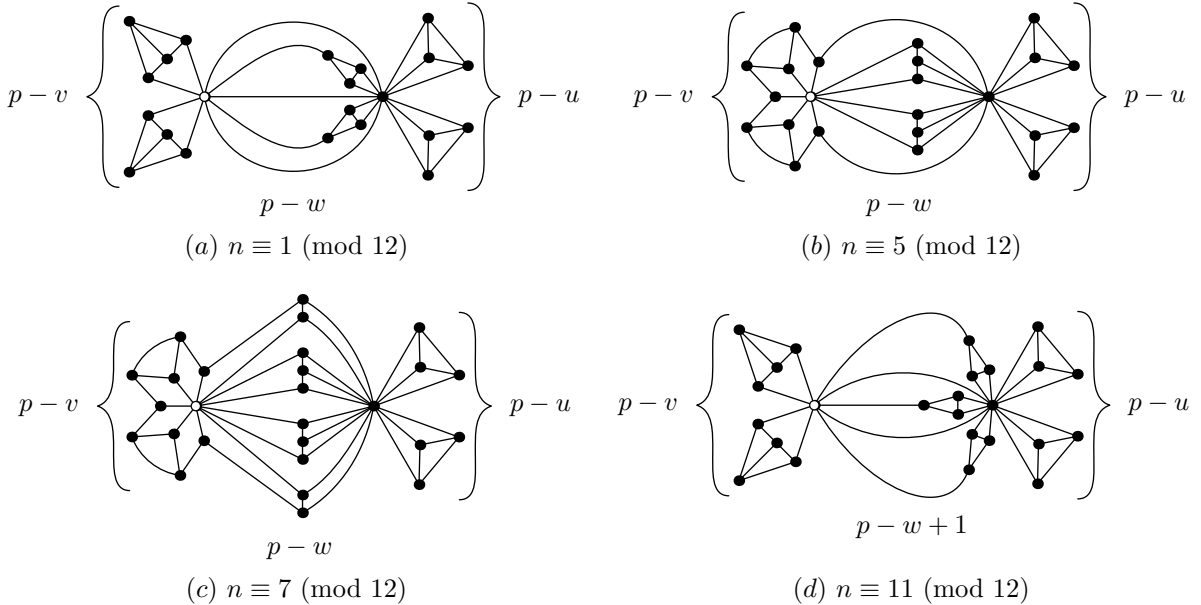

\centering
\begin{subfigure}[b]{0.45\textwidth}
\centering
\includegraphics{images/233_generic-101.mps}
\caption*{$(a)\ n\equiv 1\pmod{12}$}
\end{subfigure}
\begin{subfigure}[b]{0.45\textwidth}
\centering
\includegraphics{images/233_generic-105.mps}
\caption*{$(b)\ n\equiv 5\pmod{12}$}
\end{subfigure}
\vskip 0.3cm
\begin{subfigure}[b]{0.45\textwidth}
\centering
\includegraphics{images/233_generic-107.mps}
\caption*{$(c)\ n\equiv 7\pmod{12}$}
\end{subfigure}
\begin{subfigure}[b]{0.45\textwidth}
\centering
\includegraphics{images/233_generic-111.mps}
\vspace{0.27cm}
\caption*{$(d)\ n\equiv 11\pmod{12}$}
\end{subfigure}
\caption{Dessins $(2^r,a \sep 3^s,b \sep 3^t,c)$: triangle case.}
\label{fig:233-generic-T}
\end{figure}

\begin{figure}[ht]
\centering
\begin{subfigure}[b]{0.45\textwidth}
\centering
\includegraphics{images/233_generic-201.mps}
\caption*{$(a)\ n\equiv 1\pmod{12}$}
\end{subfigure}
\begin{subfigure}[b]{0.45\textwidth}
\centering
\includegraphics{images/233_generic-205.mps}
\caption*{$(b)\ n\equiv 5\pmod{12}$}
\end{subfigure}
\vskip 0.3cm
\begin{subfigure}[b]{0.45\textwidth}
\centering
\includegraphics{images/233_generic-207.mps}
\caption*{$(c)\ n\equiv 7\pmod{12}$}
\end{subfigure}
\begin{subfigure}[b]{0.45\textwidth}
\centering
\includegraphics{images/233_generic-211.mps}
\caption*{$(d)\ n\equiv 11\pmod{12}$}
\end{subfigure}
\caption{Dessins $(2^r,a \sep 3^s,b \sep 3^t,c)$: non-triangle case-3.}
\label{fig:233-generic-NT-3}
\end{figure}

\begin{figure}[ht]
\centering
\begin{subfigure}[b]{0.9\textwidth}
\centering
\includegraphics[align=c]{images/332_generic-301.mps}
\caption*{$(a)\ n\equiv 1\pmod{12}$}
\end{subfigure}
\begin{subfigure}[b]{0.9\textwidth}
\centering
\includegraphics[align=c]{images/332_generic-305.mps}
\caption*{$(b)\ n\equiv 5\pmod{12}$}
\end{subfigure}
\begin{subfigure}[b]{0.9\textwidth}
\centering
\includegraphics[align=c]{images/332_generic-307.mps}
\caption*{$(c)\ n\equiv 7\pmod{12}$}
\end{subfigure}
\begin{subfigure}[b]{0.9\textwidth}
\centering
\includegraphics[align=c]{images/332_generic-311.mps}
\caption*{$(d)\ n\equiv 11\pmod{12}$}
\end{subfigure}
\caption{Dessins $(2^r,a \sep 3^s,b \sep 3^t,c)$: non-triangle case-2 (dual dessins).}
\label{fig:233-generic-NT-2}
\end{figure}

\clearpage

\begin{algorithm}
\label{algo:main}
Input: a genus-$0$ passport of the form $(2^r,a \sep 3^s,b \sep 3^t,c)$ with $b\le c$.
\begin{itemize}
\item[Step 1.] Compute $u=\lfloor a/2\rfloor$, $v=\lfloor b/3\rfloor$, $w=\lfloor c/3\rfloor$, and $p=\lfloor(u+v+w)/2\rfloor$.
\item[Step 2.] Construct the dessin using the appropriate figure:
{\setlength{\leftmargini}{2.5cm}
\begin{itemize}
\item[(T)] $u\le p$, $v\le p$, $w\le p$: Fig.~\ref{fig:233-generic-T}, parameters $p-u$, $p-v$, $p-w$, $p-w+1$;
\item[(NT-3)] $w>p$: Fig.~\ref{fig:233-generic-NT-3}, parameters $u$, $v$, $w-p$, $w-p-1$;
\item[(NT-2)] $u>p$: Fig.~\ref{fig:233-generic-NT-2}, parameters $v$, $w$, $u-p$, $u-p-1$, $u-p-2$.
\end{itemize}
}
\item[Step 3.] In case NT-2, take the dual dessin.
\end{itemize}
\end{algorithm}

\noindent
{\it Proof of correctness of Algorithm~\ref{algo:main}.}
All parameters used by the algorithm in the respective constructions are non-negative, with the sole exception of the parameter $u-p-2$ in case (NT-2) with $n\equiv 11\pmod{12}$, where it equals $-1$. This is nevertheless a valid construction, as noted in Remark~\ref{rem:negative-connector}.

Thus the algorithm correctly produces a dessin.
It remains to verify that the resulting dessin has the required passport.
Behold! \hfill $\square$


\section{Realizability of $(2^r,a \sep 3^s \sep 3^t,b,c)$}
\label{section:realize-b}

\subsection{Additional Terminals}

We now have essentially everything needed to prove the realizability of genus-$0$ passports of the form
\begin{equation}
\pi = (2^r,a \sep 3^s \sep 3^t,b,c), \quad \text{where $2\nmid a$, $3\nmid b$, $3\nmid c$.}
\label{eq:passport-b}
\end{equation}

We introduce additional terminals containing one face of degree not divisible by $3$; see Fig.~\ref{fig:233-terminals-F}. Connecting terminals $A^i_1$, $A^i_2$ and $F^j_1$, $F^j_2$ via the standard $(2,3,3)$-connectors $C$, we obtain certain dessins with passports of the form~(\ref{eq:passport-b}); see Fig.~\ref{fig:233-generic2-NT-3}.

\begin{figure}[ht]
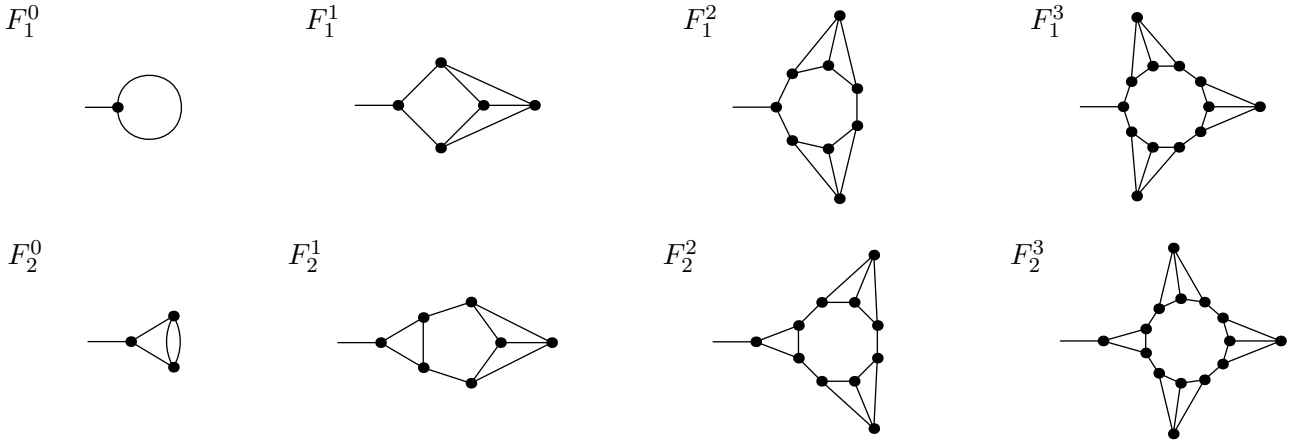

\centering
\begin{tabular}{ccccccc}
\rlap{\raisebox{2.6\totalheight}{$F^0_1$}}
\hspace{0.8cm}
\includegraphics[align=c]{images/233_terminals-40.mps}
&
\hspace{0.5cm}
&
\rlap{\raisebox{2.6\totalheight}{$F^1_1$}}
\hspace{0.4cm}
\includegraphics[align=c]{images/233_terminals-41.mps}
&
\hspace{0.5cm}
&
\rlap{\raisebox{2.6\totalheight}{$F^2_1$}}
\hspace{0.4cm}
\includegraphics[align=c]{images/233_terminals-42.mps}
&
\hspace{0.8cm}
&
\rlap{\raisebox{2.6\totalheight}{$F^3_1$}}
\hspace{0.4cm}
\includegraphics[align=c]{images/233_terminals-43.mps}
\\
\\
\rlap{\raisebox{2.6\totalheight}{$F^0_2$}}
\hspace{0.8cm}
\includegraphics[align=c]{images/233_terminals-50.mps}
&
\hspace{0.5cm}
&
\rlap{\raisebox{2.6\totalheight}{$F^1_2$}}
\hspace{0.4cm}
\includegraphics[align=c]{images/233_terminals-51.mps}
&
\hspace{0.5cm}
&
\rlap{\raisebox{2.6\totalheight}{$F^2_2$}}
\hspace{0.4cm}
\includegraphics[align=c]{images/233_terminals-52.mps}
&
\hspace{0.5cm}
&
\rlap{\raisebox{2.6\totalheight}{$F^3_2$}}
\hspace{0.4cm}
\includegraphics[align=c]{images/233_terminals-53.mps}\\
\end{tabular}
\caption{$(2,3,3)$-terminals for $c\equiv 1,2\pmod{3}$ (face).}
\label{fig:233-terminals-F}
\end{figure}


\subsection{The Construction Algorithm}

The edge count of a dessin with passport~(\ref{eq:passport-b}) is odd and divisible by $3$, so $n\equiv 3,9\pmod{12}$ and $b+c\equiv 0\pmod{3}$. Write
$$
a = 2u+1, \quad b = 3v + 1, \quad c = 3w + 2, \quad \text{where} \quad u,v,w\ge 0.
$$
The Riemann--Hurwitz formula then gives
$$
2 = r + 1 + s + t + 2 - n = \frac{n-a}{2} + \frac{n}{3} + \frac{n-b-c}{3} + 3 - n = \frac{n-3}{6} - u - v - w + 2,
$$
$$
u+v+w = \frac{n-3}{6}.
$$
Consequently, $u+v+w$ is even when $n\equiv 3\pmod{12}$ and odd when $n\equiv 9\pmod{12}$.

Define the ``integer semiperimeter''
$$
p =
\begin{cases}
\displaystyle
\frac{u+v+w}{2}, & \text{for}\ n\equiv 3\pmod{12};\\[4pt]
\displaystyle
\frac{u+v+w-1}{2}, & \text{for}\ n\equiv 9\pmod{12}.
\end{cases}
$$

\begin{algorithm}
\label{algo:main2}
Input: a genus-$0$ passport of the form $(2^r,a \sep 3^s \sep 3^t,b,c)$ with $b\equiv 1\pmod{3}$, $c\equiv 2\pmod{3}$.
\begin{itemize}
\item[Step 1.] Compute $u=\lfloor a/2\rfloor$, $v=\lfloor b/3\rfloor$, $w=\lfloor c/3\rfloor$, and $p=\lfloor(u+v+w)/2\rfloor$.
\item[Step 2.] Construct the dessin using the appropriate figure:
{\setlength{\leftmargini}{2.5cm}
\begin{itemize}
\item[(T)] $u\le p$, $v\le p$, $w\le p$: Fig.~\ref{fig:233-generic2-T}, parameters $p-u$, $p-v$, $p-w$, $p-v+1$;
\item[(NT-3)] $w>p$: Fig.~\ref{fig:233-generic2-NT-3}, parameters $u$, $v$, $w-p$, $w-p-1$;
\item[(NT-2)] $u>p$: Fig.~\ref{fig:233-generic2-NT-2}, parameters $v$, $w$, $u-p-1$.
\end{itemize}
}
\item[Step 3.] In case NT-2, take the dual dessin.
\end{itemize}
\end{algorithm}

\begin{remark}
Our dessins have two faces whose degrees are not divisible by $3$: one outer face (whose center we may place at infinity) and one inner face. In Fig.~\ref{fig:233-generic2-T}$(a)$ the inner face has degree congruent to $1$ mod $3$, i.e.\ it is the face of degree $b$. In Fig.~\ref{fig:233-generic2-T}$(b)$ the inner face has degree congruent to $2$ mod $3$, i.e.\ it is the face of degree $c$.
\end{remark}

\begin{remark}
In the (NT-3) case, Algorithm~\ref{algo:main2} uses four distinct constructions, just as Algorithm~\ref{algo:main} does. In the (T) and (NT-2) cases, however, only two constructions are needed. This can be explained as follows.
\begin{enumerate}
\item
In case (T), the constructions in Fig.~\ref{fig:233-generic2-T}$(a)$ and $(b)$ do not require the outer face degree to exceed the inner face degree, so each picture actually covers two sub-cases: $b<c$ and $b>c$.
\item
In case (NT-2) we provide constructions for the dual dessin, which has two black vertices of degrees $b$ and $c$. Again, a single picture covers both $b<c$ and $b>c$.
\end{enumerate}
\end{remark}

\noindent
{\it Proof of correctness of Algorithm~\ref{algo:main2}.}
All parameters used by the algorithm in the respective constructions (such as $p-u$, $w-p-1$, etc.) are non-negative.
Thus the algorithm correctly produces a dessin.
Checking the degrees of vertices and faces confirms that the passport of the resulting dessin matches the given one.
Behold!\hfill $\square$

\begin{figure}[htp]
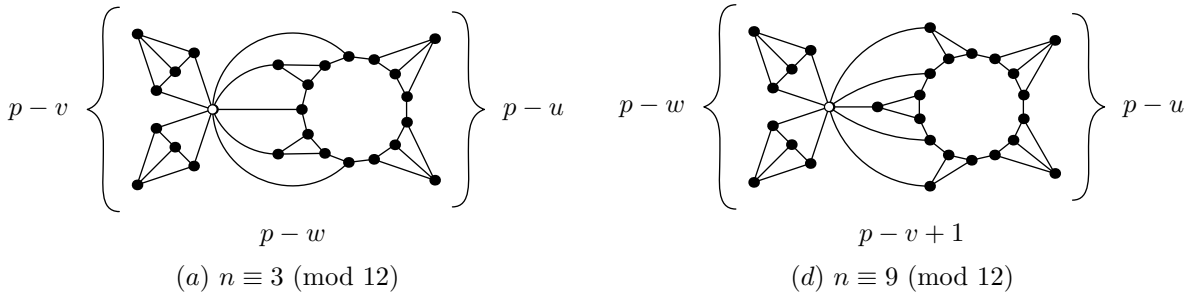

\centering
\begin{subfigure}[b]{0.45\textwidth}
\centering
\includegraphics{images/233_generic-403.mps}
\caption*{$(a)\ n\equiv 3\pmod{12}$}
\end{subfigure}
\begin{subfigure}[b]{0.45\textwidth}
\centering
\includegraphics{images/233_generic-409.mps}
\caption*{$(d)\ n\equiv 9\pmod{12}$}
\end{subfigure}
\caption{Dessins $(2^r,a \sep 3^s \sep 3^t,b,c)$: triangle case.}
\label{fig:233-generic2-T}
\end{figure}

\begin{figure}[htp]
\centering
\begin{subfigure}[b]{0.45\textwidth}
\centering
\includegraphics{images/233_generic-503.mps}
\caption*{$(a)\ n\equiv 3\pmod{12}$, $b < c$.}
\end{subfigure}
\begin{subfigure}[b]{0.5\textwidth}
\centering
\includegraphics{images/233_generic-513.mps}
\caption*{$(b)\ n\equiv 3\pmod{12}$, $b > c$.}
\vspace{0.06cm}
\end{subfigure}
\begin{subfigure}[b]{0.45\textwidth}
\centering
\includegraphics{images/233_generic-509.mps}
\caption*{$(c)\ n\equiv 9\pmod{12}$, $b < c$.}
\end{subfigure}
\begin{subfigure}[b]{0.5\textwidth}
\centering
\includegraphics{images/233_generic-519.mps}
\vspace{0.06cm}
\caption*{$(d)\ n\equiv 9\pmod{12}$, $b > c$}
\end{subfigure}
\caption{Dessins $(2^r,a \sep 3^s,b \sep 3^t,c)$: non-triangle case-3.}
\label{fig:233-generic2-NT-3}
\end{figure}

\begin{figure}[t]
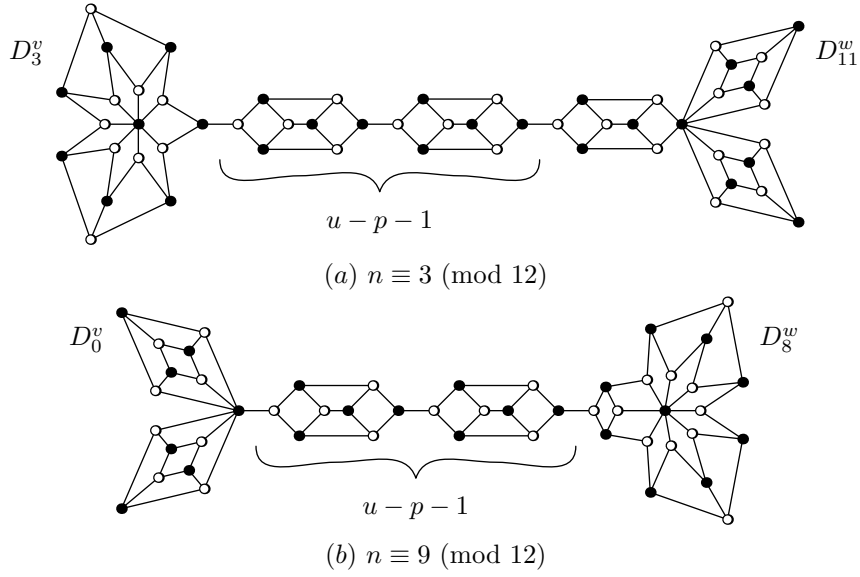

\centering
\begin{subfigure}[b]{0.9\textwidth}
\centering
\includegraphics[align=c]{images/332_generic-603.mps}
\caption*{$(a)\ n\equiv 3\pmod{12}$}
\end{subfigure}
\begin{subfigure}[b]{0.9\textwidth}
\centering
\includegraphics[align=c]{images/332_generic-609.mps}
\caption*{$(b)\ n\equiv 9\pmod{12}$}
\end{subfigure}
\caption{Dessins $(2^r,a \sep 3^s \sep 3^t,b,c)$: non-triangle case-2 (dual dessins).}
\label{fig:233-generic2-NT-2}
\end{figure}

\section*{Acknowledgments}

The authors are grateful to Fedor Pakovich for valuable discussions and useful comments.
The research of E. Kreines was  supported by the ISF Grant  1092/22.


\begin{thebibliography}{99}

\bibitem{Adr-2007}
N. M. Adrianov,
On planes trees with a prescribed number of valency set realizations,
J. Math. Sci., 158 (2009), no.~1, 5--10.

\bibitem{AdrPakZvo-2020}
N. M. Adrianov, F. Pakovich, A. K. Zvonkin,
\textit{Davenport-Zannier Polynomials and Dessins d'Enfants},
AMS Mathematical Surveys and Monographs, {\bf 249}, Providence, RI (2020).

\bibitem{BarPetr-2024} F. Baroni, C. Petronio,
Solution of the Hurwitz problem with a length-2 partition,
Illinois J. Math., 68 (2024), no.~3, 493-–529.

\bibitem{Bel-1979}
G. V. Belyi,
On Galois extensions of a maximal cyclotomic field,
Math. USSR Izvestija, 14 (1980), no.~2, 247--256.

\bibitem{BCHS-1965}
B. J. Birch, S. Chowla, M. Hall, Jr., A. Schinzel,  On the difference $x^3-y^2$,
Norske Vid. Selsk. Forh. (Trondheim), 38 (1965), 65--69.

\bibitem{Caj-1910}
F. Cajori,
\textit{A History Of Elementary Mathematics With Hints On Methods Of Teaching},
pp. viii + 304. (New York: The Macmillan Company. London: Macmillan and Co., Ltd., 1896).
Russian translation with appendixes by I. Yu. Timchenko, Mathesis, Odessa, 1910.

\bibitem{Hur-1891}
A. Hurwitz,
Ueber Riemann'sche Fl\"achen mit gegebenen Verzweigungspunkten,
Math. Ann. 39 (1891), 1--60.

\bibitem{Gro-1984}
A. Grothendieck,
Sketch of a~programme (Esquisse d'un programme),
in: \textit{Geometric Galois Actions I} (L. Schneps, P. Lochak eds.),
London Math.\ Soc.\ Lecture Notes Series, 242,
Cambridge Univ.\ Press, Cambridge (1997), 5--48;
English translation: 243--283.

\bibitem{EKS-1984}
A. L. Edmonds, R. S. Kulkarni, R. E. Stong,
Realizability of branched coverings of surfaces,
Trans. Amer. Math. Soc., 282 (1984), 773--790.

\bibitem{KLN-2018}
J. König, A. Leitner, D. Neftin,
Almost-regular dessins d'enfant on a torus and sphere,
Topology and its Applications, 243 (2018), 78--99.

\bibitem{LanZvo-2004} S. K. Lando, A. K. Zvonkin, {\it Graphs on surfaces and their applications},
with an appendix by D. Zagier.
Encycl. of Math. Sciences, {\bf 141}, Springer, 2004.

\bibitem{MengWeiZhou-2026} Y. Meng, Z. Wei, C. Zhou,
Existence of branched covers $S^2\to S^2$ with prescribed branching data.
Topology and its Applications, 384 (2026), 109793.

\bibitem{Pak-2009}
F. Pakovich,
Solution of the Hurwitz problem for Laurent polynomials,
Journal of Knot Theory and Its Ramifications, 18 (2009), no.~2, 271--302.

\bibitem{Pak-2024}
F. Pakovich,
Hurwitz existence problem and fiber products,
arXiv:2408.10874.

\bibitem{PakZvo-2014}
F. Pakovich, A. K. Zvonkin,
Minimum degree of the difference of two polynomials over $Q$, and weighted plane trees,
Selecta Math. (N.S.), 20 (2014), no.~4, 1003--1065.

\bibitem{PascPetr-2012}
M. A. Pascali, C. Petronio, Branched covers of the sphere and the prime-degree conjecture,
Ann. Mat. Pura Appl., 191 (2012), 563--594.

\bibitem{PervPetr-2006}
E. Pervova, C. Petronio,
On the existence of branched coverings between surfaces with prescribed branch data. I,
Algebraic \& Geometric Topology, 6 (2006), 1957–1985.

\bibitem{PervPetr-2008}
E. Pervova, C. Petronio,
On the existence of branched coverings between surfaces with prescribed branch data. II,
Journal of Knot Theory and Its Ramifications, 17 (2008), 787–816.

\bibitem{Petr-2020} C. Petronio, The Hurwitz existence problem for surface branched covers,
Winter Braids Lecture Notes, 7 (2020), Course no II, 1--43.

\bibitem{Plo-2007}
K. Plofker, Mathematics in India in
\textit{The Mathematics of Egypt, Mesopotamia, China, India, and Islam: A Sourcebook}
/ ed. by V. J. Katz. -- Princeton : Princeton University Press, 2007, 385--514.

\bibitem{Plo-2009}
K. Plofker,
\textit{Mathematics in India}, 384 p.
Princeton: Princeton University Press, 2009.

\bibitem{Sto-1981}
W. W. Stothers, Polynomial identities and Hauptmoduln,
Quart. J. Math. Oxford, ser. 2, 32 (1981), no.~127, 349--370.

\bibitem{Thom-1965}
R. Thom,
L'\'equivalence d'une fonction diff\'erentiable et d'un polynome,
Topology, 3 (1965), suppl.~2, 297--307.

\bibitem{Tutte-1964}
W. T. Tutte,
Planted plane trees with a given partition,
Amer. Math. Monthly, 71:3 (1964), 272--277.

\bibitem{WeiWuXu-2024} Z. Wei, Y. Wu, B. Xu, A note on rational maps with three branching points on the Riemann sphere,
arXiv:2401.06956, 2024.

\bibitem{Zheng-2006}
H. Zheng,
Realizability of branched coverings of $S^2$,
Topology and its Applications, 153 (2006), 2124--2134.

\end{thebibliography}
\end{document}